\documentclass[11pt,a4paper]{article}

\usepackage{amsfonts}

\usepackage{amssymb,amsthm, amsmath,graphicx,bm,ebezier, mathtools}

\usepackage{hyperref}
\usepackage{color,enumerate}
\usepackage{url,lineno}
\usepackage{soul}
\usepackage[normalem]{ulem}

\usepackage[ruled,vlined]{algorithm2e}

\newtheorem{maintheorem}{Theorem}

\newtheorem{theorem}{Theorem}

\newtheorem{proposition}[theorem]{Proposition}
\newtheorem{corollary}[theorem]{Corollary}
\newtheorem{lemma}[theorem]{Lemma}

\newtheorem{remark}[theorem]{Remark}

\def\N{\mathbb{N}}

\newcommand{\bN}{\mathbb{N}}

\newcommand{\ord}{\mbox{\rm ord}}

\newenvironment{cambio rojo}{\color{red}}{}	
\newenvironment{cambio azul}{\color{blue}}{}	

\begin{document}

\title{Conductors of Abhyankar-Moh semigroups of even degrees}

\author{Evelia R. GARC\'IA BARROSO, Juan Ignacio GARC\'IA-GARC\'IA, \\ Luis Jos\'e SANTANA S\'ANCHEZ and Alberto VIGNERON-TENORIO}
\date{}

\maketitle

\begin{abstract} 
In their paper on the embeddings of the line in the plane, Abhyankar and Moh proved an important inequality, now known as the Abhyankar-Moh inequality, which can be stated in terms of the semigroup associated with the branch at infinity of a plane algebraic curve. Barrolleta, Garc\'{\i}a Barroso and P\l oski studied the semigroups of integers satisfying the Abhyankar–Moh inequality and call them Abhyankar-Moh semigroups. They described such semigroups with the maximum conductor. In this paper we prove that all possible conductor values are achieved for the Abhyankar-Moh semigroups of even degree. Our proof is constructive, explicitly describing families that achieve a given value as its conductor.

{\small
\noindent {\it Keywords:} Abhyankar-Moh inequality, Abhyankar-Moh semigroups, conductor, gluing of semigroups.

\noindent 2020 {\it Mathematics Subject Classification:} 20M14, 14H20.
}

\end{abstract}

\section{Introduction}

Suzuki on the one hand  (see \cite{Suzuki}) and Abhyankar and Moh on the other (see \cite{Abhyankar-Moh}) proved independently  that the affine line can be embedded in a unique way, up to ambient automorphisms, in the affine plane.  Let us indicate some details of this fact. Let $K$ be an algebraically closed field of arbitrary characteristic. A polynomial mapping $\sigma_{p,q}:K\longrightarrow K^2$ defined as $\sigma_{p,q}(x,y)=(p(x,y),q(x,y))$ is a {\it polynomial embedding of the line} $K$ if there is a polynomial map $g:K^2\longrightarrow K$ such that $g(p(t), q(t)) = t$ in $K[t]$. This is equivalent to the equality $K[p(t), q(t)] =K[t]$.\\ An affine curve $C\subset K^2$ is an 
{\it embedded line} if there exists a polynomial embedding $\sigma_{p,q}$  such that
$\sigma_{p,q}(K)=C$.  It is easy
to check that any embedded line is an irreducible affine curve. \\Let $C$ be
an embedded line with a minimal equation $f(x,y)=0$. After \cite{Abhyankar-Moh}, the curve  $C$ has only {\it one place at infinity}, that is the closure $\overline C$ of $C$ in the proyective plane has only one point $O_{\infty}$ on the line at infinity and it is unibranch at $O_{\infty}$ (the polynomial $f(x,y)$ is irreducible as an element of the formal power series ring $K[[x,y]]$). In this case, associated with  $\overline C$  we have a numerical semigroup $S(\overline C)$ consisting of zero and all intersection numbers of $\overline C$ with all algebroid curves not having $\overline C$ as an irreducible component. By the Bresinsky–Angerm\" uller Theorem (\cite{Bresinsky} for zero characteristic and \cite{Angermuller} for arbitrary characteristic) there exists a (unique) 
sequence $(v_0, \ldots,v_h)$, called the {\it characteristic  at infinity} of $C$, generating $S(\overline C)$ where $v_0$ is the degree of $C$.\\ Assume that $C$ is an affine irreducible curve of degree greater than 1 with one branch
at infinity and let $(v_0, \ldots,v_h)$ be its characteristic at infinity. Suppose that $\gcd(\deg C, \ord_{O_{\infty}} \overline C)\not\equiv 0$ (mod char $K$). Then, after \cite{Abhyankar-Moh}
\begin{equation}\label{InAM}
    \gcd (v_0,\ldots,v_{h-1})v_h<v_0^2.
\end{equation}
The condition $\gcd(\deg C, \ord_{O_{\infty}} \bar C)\not\equiv 0$ (mod char $K$) is automatically satisfied when the characteristic of $K$ is zero but otherwise it is essential.\\
The inequality \eqref{InAM} is called {\it the Abhyankar-Moh inequality}. Originally this inequality appears linked to the Puiseux expansion of the given branch at the infinite place (see \cite[equality (35)]{Abhyankar-Moh}). The semigroups of integers associated with branches and
satisfying the Abhyankar–Moh inequality are called {\it Abhyankar-Moh semigroups} of degree $v_0$ (the order of the branch). The Abhyankar-Moh semigroups were studied in \cite{Ba-GB-P}, where these semigroups
with maximum possible conductor, which is equal to $(v_0-2)(v_0-1)$, were described. Later, in \cite{GB-Gw-P}, a geometric interpretation of the branches with Abhyankar-Moh semigroup and maximum possible conductor was given. It is well known that the conductor of an Abhyankar-Moh semigroup of degree $v_0$ is an even integer  belonging to the interval $[v_0-1,(v_0-1)(v_0-2)]$. Our main result is:
\begin{maintheorem}\label{main}
Let $n>2$ be an even natural number. For any even number $c$ with $n-1\leq c\leq (n-1)(n-2)$, there is an Abhyankar-Moh semigroup of degree $n$ and conductor equals $c$.
\end{maintheorem}

In order to prove  Theorem \ref{main} we start by computing in Section \ref{computing}, in an algorithmically way, the set of all Abhyankar-Moh semigroups of a fixed degree (see Algorithm 2). Section \ref{proof} is devoted to the proof of Theorem A. This proof is constructive and the algorithms of Section  \ref{computing} play a fundamental role in it.

From the computational point of view, it remains, as an open question, to determine the values of conductors reached by the Abhyankar-Moh semigroups of odd degree, but it seems that this poses new computational challenges. From the geometric point of view, the next step would be to geometrically characterize the branches with Abhyankar-Moh  semigroups of even degree, studied in this paper, following the line given in \cite{GB-Gw-P} for those with the maximum possible conductor.

\section{Preliminaries}
A numerical semigroup $S$ is an additive submonoid of $\N$ with finite complement in $\N$. It is well known that numerical semigroups are finitely generated, and their minimal generating sets are unique. The largest integer in $\N\setminus S$ is called the {\it Frobenius number} of $S$. The {\it conductor} of $S$ is the  Frobenius number of $S$ plus 1.

Given a finite set $A=\{a_1,\ldots ,a_t\}\subset \N$, $A$ is the generating set of $S$ when $S=\N a_1 +\dots +\N a_t$. In this work, when $A$ is the minimal generating set of $S$, we assume that $a_1<\cdots <a_t$.  Besides, $S=\langle A\rangle$ means that $A$ is the minimal set of generators of $S$. The cardinality of the minimal generating set of $S$ is called the {\it embedding dimension} of $S$. 

A sequence of positive integers $(v_0,\ldots,v_h)$ is called a
{\em characteristic sequence} if it satisfies the following two
properties:
\begin{enumerate}
\item[\hypertarget{cs1}{\rm{(CS1)}}] Put $e_k=\gcd (v_0,\ldots,v_k)$ for $0\leq k \leq  h$. Then $e_k<e_{k-1}$ for $1\leq k \leq  h$
and $e_h=1$.
\item[\hypertarget{cs2}{\rm{(CS2)}}] $e_{k-1}v_k<e_kv_{k+1}$ for $1\leq k \leq  h-1$.
\end{enumerate}

\noindent We put  $n_k= \frac{e_{k-1}}{e_k}$ for $1\leq k\leq h$. Therefore,
$n_k>1$ for $1\leq k \leq  h$ and $n_{h}=e_{h-1}$. If $h=0$, the only
characteristic sequence is $(v_0)=(1)$. If $h=1$, the sequence
$(v_0,v_1)$ is a characteristic sequence if and only if  $\gcd
(v_0,v_1)=1$. Property \hyperlink{cs2}{\rm{(CS2)}} plays a role  if and only if
$h\geq 2$.

\begin{lemma}(\cite[Lemma 1.1]{Ba-GB-P})
\label{llll}
Let $(v_0,\ldots,v_h)$ be a characteristic sequence with $h\geq 2$. Then,
\begin{enumerate}
\item[\hbox{\rm (i)}] $v_1<\cdots<v_h$ and $v_0<v_2$.
\item[\hbox{\rm (ii)}] Let $v_1<v_0$. If $v_0\not\equiv 0$ \hbox{\rm (mod $v_1$) } then $(v_1,v_0,v_2,\ldots,v_h)$ is a cha\-rac\-teristic sequence. If  $v_0\equiv 0$ \hbox{\rm (mod $v_1$) } then $(v_1,v_2,\ldots,v_h)$ is a cha\-rac\-teristic sequence.
\end{enumerate}
\end{lemma}

A semigroup $S\subseteq \bN$  is {\em strongly increasing} (SI-semigroup) if $S\neq \{0\}$ and it is generated by a characteristic sequence, that is, $S=\N v_0+\cdots +\N v_h$. We will denote  by $S=S(v_0,v_1,\ldots,v_h)$ the numerical semigroup generated by the characteristic sequence $(v_0,v_1,\ldots,v_h)$.

A SI-semigroup $S(v_0,v_1,\ldots,v_h)\subseteq \N$ is  an {\em Abhyankar-Moh semigroup (A-M semigroup)} of degree $n=v_0 >1$ if it satisfies the {\em Abhyankar-Moh inequality}
\[ e_{h-1}v_h< n^2.\]
Observe that the semigroup $\mathbb N=S(n,1)$ for any $n\in \mathbb N$, so $\mathbb N$ is an Abhyankar-Moh semigroup of any degree.

The conductor of the A-M semigroup $S=S(v_0,v_1,\ldots,v_h)$ is 
\begin{equation}
\label{conductor}
c(S)=\sum_{i=1}^h(n_i-1)v_i-v_0+1,
\end{equation}
where $n_i=\frac{e_{i-1}}{e_i}$ for $i\in \{1,\ldots,h\}$. Moreover $c(S)$ is an even integer (see \cite[Proposition 1.2]{Ba-GB-P}).

\begin{remark}\label{v0>v1}
If $S=S(v_0,v_1,\ldots,v_h)$ is an A-M
semigroup of degree $n=v_0>1$  then $v_0>v_1$, since
$e_0v_1<e_1v_2<\cdots<e_{h-1}v_h<n^2=v_0^2$. 
\end{remark}

Let $S=S(v_0,v_1,\ldots,v_h)$ be an A-M semigroup of degree $n=v_0>1$. If $c(S)$ is the conductor of $S$ then, by \cite[Theorem 2.2]{Ba-GB-P},
\begin{equation}\label{cota1}
c(S)\leq (n-1)(n-2).
\end{equation}
Moreover,
\begin{equation}\label{Th:maximal}
c(S)=(n-1)(n-2) \text{ if and only if }v_k=\frac{n^2}{e_{k-1}}-e_k \text{ for }1\leq k \leq h.
\end{equation}
Hence, if $S\neq \mathbb N$ is an  A-M semigroup
of degree $n>1$ its conductor $c(S)$ is an even integer
number verifying the inequalities
\begin{equation}
\label{ineq}
n-1\leq c(S)\leq (n-1)(n-2).
\end{equation}
By Remark \ref{v0>v1} we get that the only A-M
semigroup of degree $2$ is generated by the characteristic sequence
$(2,1)$. Such semigroup  achieves the upper bound for the conductor,
given in \eqref{cota1}.

After \cite[Proposition 1.2]{Ba-GB-P}, if $(v_0,v_1,\ldots ,v_h)$ is a characteristic sequence and $S=S(v_0,v_1,\ldots ,v_h)$, then, $\{\min(v_0,v_1),v_2,\ldots ,v_{h-1},v_h\}$ is a subset of the minimal generating set of $S$, and $v_2\notin \mathbb N v_0+ \mathbb N v_1$. Furthermore, for every $w\in S$, $(v_0,v_1,\ldots ,v_i,w,v_{i+1},\ldots ,v_h)$ is not a characteristic sequence, for every $i=0,\ldots ,h-1$.

Let $S=S(v_0,v_1,\ldots ,v_h)$ be a SI-semigroup. By Lemma \ref{llll}, the characteristic sequences generating $S$ are:
\begin{itemize}
    \item $(v_0,v_1,v_2,\ldots ,v_h)$.
    \item If $v_0<v_1$, $(v_1,v_0,v_2,\ldots ,v_h)$.
    \item If $v_1<v_0$ and $v_0\not\equiv 0$ \hbox{\rm (mod $v_1$)},  $(v_1,v_0,v_2,\ldots,v_h)$.
    \item If $v_1<v_0$ and $v_0\equiv 0$ \hbox{\rm (mod $v_1$)}, $(v_1,v_2,\ldots,v_h)$, and $(v_2,v_1,\ldots,v_h)$.
    \item $(kv'_0,v'_0,v'_1,v_2,\ldots ,v_h)$ for every integer $k\in [2, v'_1/v'_0  ]$ where $v'_0=\min(v_0,v_1)$, $v_1'=\max(v_0,v_1)$ and $v'_1\not\equiv 0$ \hbox{\rm (mod $v'_0$)}.
\end{itemize}
As a consequence, we determine the characteristic sequences generating a SI-semigroup given by its minimal generating set.

Let $S=\langle a_1,\ldots ,a_t\rangle$ be a SI-semigroup. Then, from \cite[Corollary 1.4]{Ba-GB-P},  the characteristic sequences generating $S$ are:
\begin{equation}
\begin{array}{l}
\bullet\, (a_1,\ldots ,a_t),\\
\bullet\, (a_2,a_1,a_3,\ldots ,a_t),\\
\bullet  \text{ and }(ka_1,a_1,a_2,\ldots ,a_t)\text{ for every integer }k\in [2, a_2/a_1 ).
\end{array}
\end{equation}

Similarly, the possible characteristic sequences generating an A-M semigroup given by its minimal generating set can be described: let $S=\langle a_1,\ldots ,a_t\rangle$ be a SI-semigroup. If $S$ is an A-M semigroup, then
\begin{equation}\label{SqAM}
S=S(a_2,a_1,a_3,\ldots ,a_t),\text{ or }S=S(ka_1,a_1,a_2,\ldots ,a_t)
\end{equation}
for every integer $k\in \Big(\frac{\sqrt{\gcd(a_1,\ldots ,a_{t-1})a_t}}{a_1},\frac{a_2}{a_1}\Big)$. For the first case, $S$ should be an A-M semigroup of degree $a_2$, and of degree $ka_1$ for the second (see \cite[Proposition 2.1]{Ba-GB-P}).

Notice  that if you want to check whether the semigroup $\langle a_1,\ldots ,a_t\rangle$ is an A-M semigroup, you only need to check whether it is an A-M semigroup of degree $a_2$.

\section{Computing Abhyankar-Moh semigroups}
\label{computing}
The A-M semigroups minimally generated by two elements are easy to describe. Whenever $b>a>1$ are two coprime integers, $\langle b,a\rangle$ is an A-M semigroup of degree $b$. Furthermore, $\langle ka,b\rangle$ is also an A-M semigroup of degree $ka$ for every  $k\in (\sqrt{b/a},b/a )\cap \N$, with $k$ coprime with $b$.

A-M semigroups of higher embedding dimensions are completely characterized by {\it gluings}: the gluing of $S=\langle a_1,\ldots ,a_t \rangle $ and $\mathbb N$ with respect to the positive integers $d$ and $f$ with $\gcd(d,f)=1$ (see \cite[Chapter 8]{semigrupos}) is the numerical semigroup $\N da_1 +\dots +\N da_t +\N f$. We denote it by $S \oplus_{d,f}\N$. 

\begin{proposition}\label{proposition_construction_A-M}
The set $\bar S$ is an A-M semigroup with embedding dimension $t\ge 3$ if and only if $\bar S= S\oplus_{d,f}\N$ where $S=\langle b_1,\ldots ,b_{t-1}\rangle$ is an A-M semigroup of degree $m$, and $f,d$ are two coprime integers such that $dm^2>f>d\gcd(b_1,\ldots, b_{t-2})b_{t-1}$. Moreover, the degree of $\bar S$ is $dm$.
\end{proposition}

\begin{proof}

Consider $\bar S= \langle a_1,\ldots ,a_t\rangle$, and put $d= \gcd (a_1,\ldots ,a_{t-1})$, $f=a_t$, $b_i=a_i/d$ for all $i=1,\ldots ,t-1$, and the SI-semigroup $S=\langle b_1,\ldots ,b_{t-1}\rangle$. Trivially, $\bar S= S\oplus_{d,f}\N$. Since $\bar S$ is a SI-semigroup, 
\begin{eqnarray*}f&=&a_t > \gcd(a_1,\ldots ,a_{t-2})a_{t-1}/d=d\gcd(a_1/d,\ldots ,a_{t-2}/d)a_{t-1}/d\\
&=& d\gcd(b_1,\ldots, b_{t-2})b_{t-1}.
\end{eqnarray*}
By \eqref{SqAM}, the degree of $\bar S$ equals $a_2$, or to $ka_1$ for some integers $k>1$. If we assume that the degree is $a_2$, then $(a_2,a_1,a_3,\ldots ,a_t)$ is a characteristic sequence generating $\bar S$, and 
$a_2^2>\gcd(a_1,\ldots ,a_{t-1})a_t$. Thus, $(b_2,b_1,b_3,\ldots ,b_{t-1})$ is a characteristic sequence generating $S$, and 
\begin{eqnarray*}
b_2^2 &=& (a_2/d)^2>\gcd(a_1,\ldots ,a_{t-1})a_t/d^2> \gcd(a_1,\ldots ,a_{t-2})a_{t-1}/d^2 \\&=& \gcd(b_1,\ldots ,b_{t-2}) b_{t-1}.
\end{eqnarray*}
Hence, $S$ is an A-M semigroup of degree $m=b_2$.

Similarly, for the degree $ka_1$, $(kb_1,b_1,b_2,\ldots ,b_{t-1})$ is a characteristic sequence generating $S$, and $(kb_1)^2=(ka_1/d)^2> \gcd(b_1,\ldots ,b_{t-2}) b_{t-1}$, so $S$ is an A-M semigroup of degree $m=kb_1$. For the both previous possibilities, the inequality $dm^2>f$ is satisfied.

Conversely, let $S=\langle b_1,\ldots ,b_{t-1}\rangle$ be an A-M semigroup of degree $m\in \N$, and $f,d$ be two coprime integers such that $dm^2>f>d\gcd(b_1,\ldots, b_{t-2})b_{t-1}$.

Since $f>d\gcd(b_1,\ldots, b_{t-2})b_{t-1}$, $\bar S=S\oplus_{d,f}\N=\langle db_1 ,\ldots , db_{t-1},f \rangle $ is a SI-semigroup (\cite[Theorem 3]{BG-GG-VT}). Again, by \eqref{SqAM}, the possible characteristic sequences generating $S$ are $(b_2,b_1,b_3,\ldots ,b_{t-1})$ and $(kb_1,b_1,b_3,\ldots ,b_{t-1})$, with degrees $b_2$ and $kb_1$, respectively. Hence, the characteristic sequences generating $\bar S$ are $(db_2,db_1,db_3,\ldots ,db_t,f)$, and, $(kdb_1,db_1,db_2,\ldots ,db_t,f)$; and by hypothesis, we have the Abhyankar-Moh inequality $(dm)^2>df$. Thus, $\bar S$ is an A-M semigroup of degree $dm$.
\end{proof}

Let $S$ be an A-M semigroup of degree $m$, and $f>d>1$ two coprime integers such that $S\oplus_{d,f}\N$ is also an A-M semigroup of degree $dm$. Then, by \cite[Proposition 10]{delorme},
\begin{equation}\label{CondGSI}
\mathrm c (S\oplus_{d,f}\N)=d\mathrm c(S)+(d-1)(f-1),
\end{equation}
where $\mathrm c(S)$ denotes the conductor of $S$.

Denote by $M(A)$ the largest element of the minimal system of generators of the numerical semigroup $\langle A \rangle$, and $s(A)=\min(A\setminus \{\min(A)\})$, that is, the second element in $A$. Algorithm \ref{algoritmo1} computes all the A-M semigroups with conductor less than or equal to a fixed non-negative integer.
\begin{algorithm}[h]
	\BlankLine
	\KwIn{$c\in\N\setminus\{0,1\}$.}
	\KwOut{The set $\{A \mid \langle A \rangle\textrm{ is an A-M semigroup with } \mathrm c (\langle A \rangle)\leq c \}$. }
	\BlankLine
	${\cal A}\leftarrow\big\{ \{a,b\} \mid 1<a<b,\, \gcd(a,b)=1,\, ab-a-b+1\le c \big\}$\;
	\ForAll {$k\in \{2,\dots,c-1\}$}
	{
		$B\leftarrow\{ A \in {\cal A}\mid \mathrm c(\langle A \rangle)=k\}$ \;
		\ForAll{$A\in B$}
		{
			$G_{A}\leftarrow\big\{(d,f) \in \N^2 \mid \gcd(f,d)=1,\, s(A)^2>f/d> \gcd\big(A\setminus\{M(A)\}\big)\cdot M(A),\,
			c \ge \max\{dk+(d-1)(f-1),d^2(k-1)+1\}\big\}$\;
		}
		${\cal A}\leftarrow{\cal A}\cup \big\{ d A \cup \{f\} \mid (d,f)\in G_A\big\}$\;
	}
	\Return $\mathcal A$ \;
\caption{Computation of the set of A-M semigroups with conductor less than or equal to $c$.}\label{algoritmo1}
\end{algorithm}

Table \ref{tabla_conductor_hasta_18} illustrates Algorithm \ref{algoritmo1}: we collect all the A-M semigroups with conductor less that or equal to $18$. We also give the characteristic sequences associated to the given semigroups.
\begin{table}[h]
\centering
\begin{tabular}{c|l}
A-M semigroups & Characteristic sequences\\\\\hline
$\langle 2,3 \rangle$ & $\{ ( 3,2 ) \}$ \\ \hline
$\langle 2,5 \rangle$ & $\{ ( 4,2,5 ),( 5,2 ) \}$ \\ \hline
$\langle 2,7 \rangle$ & $\{ ( 4,2,7 ),( 6,2,7 ),( 7,2 ) \}$ \\ \hline
$\langle 2,9 \rangle$ & $\{ ( 6,2,9 ),( 8,2,9 ),( 9,2 ) \}$ \\ \hline
$\langle 2,11 \rangle$ & $\{ ( 6,2,11 ),( 8,2,11 ),( 10,2,11 ),( 11,2 ) \}$ \\ \hline
$\langle 2,13 \rangle$ & $\{ ( 6,2,13 ),( 8,2,13 ),( 10,2,13 ),( 12,2,13 ),( 13,2 ) \}$ \\ \hline
$\langle 2,15 \rangle$ & 
$\begin{array}{l}
    \{ ( 6,2,15 ),( 8,2,15 ),( 10,2,15 ),\\
    ( 12,2,15 ),( 14,2,15 ),( 15,2 ) \} 
\end{array}$ \\ \hline
$\langle 2,17 \rangle$ & 
$\begin{array}{l}
    \{ ( 6,2,17 ),( 8,2,17 ),( 10,2,17 ),\\
    ( 12,2,17 ),( 14,2,17 ),( 16,2,17 ),( 17,2 ) \}
\end{array}$
\\ \hline
$\langle 2,19 \rangle$ & 
$\begin{array}{l}
    \{ ( 8,2,19 ),( 10,2,19 ),( 12,2,19 ),\\
    ( 14,2,19 ),( 16,2,19 ),( 18,2,19 ),( 19,2 ) \}
\end{array}$
\\ \hline
$\langle 3,4 \rangle$ & $\{ ( 4,3 ) \}$ \\ \hline
$\langle 3,5 \rangle$ & $\{ ( 5,3 ) \}$ \\ \hline
$\langle 3,7 \rangle$ & $\{ ( 6,3,7 ),( 7,3 ) \}$ \\ \hline
$\langle 3,8 \rangle$ & $\{ ( 6,3,8 ),( 8,3 ) \}$ \\ \hline
$\langle 3,10 \rangle$ & $\{ ( 6,3,10 ),( 9,3,10 ),( 10,3 ) \}$ \\ \hline
$\langle 4,5 \rangle$ & $\{ ( 5,4 ) \}$ \\ \hline
$\langle 4,7 \rangle$ & $\{ ( 7,4 ) \}$ \\ \hline
$\langle 4,6,13 \rangle$ & $\{ ( 6,4,13 ) \}$ \\ \hline
$\langle 4,6,15 \rangle$ & $\{ ( 6,4,15 ) \}$ \\ \hline
$\langle 4,6,17 \rangle$ & $\{ ( 6,4,17 ) \}$ \\ \hline
\end{tabular}
\caption{A-M semigroups up to conductor 18.}\label{tabla_conductor_hasta_18}
\end{table}

Let $n>1$ be an integer. A sequence of integers $(d_0,
\ldots, d_h)$ will be called a {\em sequence of divisors of } $n$ if
$d_i$ divides $d_{i-1}$ for $1\leq i \leq h$ and
 $n=d_0>d_1>\cdots>d_{h-1}>d_h=1$. In particular if $(v_0,v_1,\ldots,v_h)$ is a
characteristic sequence then $(e_0,\ldots,e_h)$ is a sequence of
divisors of $n=v_0$, where $e_i=\gcd(v_0,\ldots ,v_i)$. In this case we will say that $(e_0,\ldots,e_h)$ is the sequence of divisors associated with $(v_0, \ldots, v_h)$.

Using sequences of divisors, in \cite[Proposition 2.3]{Ba-GB-P} it was proved that
\begin{equation}\label{eq:divisors}
  \left(n, n-d_1,
\frac{n^2}{d_{1}}-d_2, \ldots, \frac{n^2}{d_{i-1}}-d_i,  \ldots,
\frac{n^2}{d_{k-1}}-1\right)
\end{equation}
is a characteristic sequence and the semigroup generated by it is an A-M semigroup  of degree $n$ and conductor $(n-1)(n-2)$ (which is the maximal possible conductor after \eqref{cota1}). Inspired by this idea, using sequences of divisors, we  introduce an algorithm for computing all the A-M semigroups for a given degree (Algorithm \ref{algoritmo2}). The following proposition is the key for providing this algorithm.

\begin{proposition}\label{AM-alg2}
Let $n\ge 2$ be an integer and $D=(d_0,\ldots,d_h)$ be a sequence of divisors of $n=d_0$. Then, the characteristic sequence of any A-M semigroup with degree $n$ and sequence of divisors equals $D$ is of the form
$(n,d_1k_1,\ldots,d_hk_h)$, with $1\le k_1\le \frac{d_0}{d_1}-1$, $d_{i-2}k_{i-1}+1\le d_{i}k_{i}\le  \frac{d_0^2}{d_{i-1}}-d_{i}$ for any $i=2,\ldots ,h$, and $\gcd\big(\frac{d_{i-1}}{d_i},k_i\big)=1$ for $i=1,\ldots ,h$.
\end{proposition}

\begin{proof}
Notice that the condition $\gcd\big(\frac{d_{i-1}}{d_i},k_i\big)=1$ for $i=1,\ldots ,h$, guarantees that the characteristic sequence $(n,d_1k_1,\ldots,d_hk_h)$ has associated sequence of divisors equals to $D$. It remains to check the bounds on the $k_i$. The case $h=1$ is trivially verified. Suppose $h\geq 2$. Let $v_0=d_0$, and $v_i=d_ik_i$ for any $i=1,\ldots ,h$, and assume that $(v_0,\ldots ,v_h)$ is the characteristic sequence of an A-M semigroup $S$ of degree $n=v_0$, and associated sequence of divisors $(d_0,\ldots,d_h)$. Since A-M semigroups are SI-semigroups, by the definition of $v_i$, $d_{i-1}k_i<d_{i+1}k_{i+1}$ for all $i\in \{1,\ldots, h-1\}$. Hence, $d_{i-1}k_i+1\le d_{i+1}k_{i+1}$. Moreover, since $1<v_1<v_0$, $1\le k_1\le \frac{d_0}{d_1}-1$, and since $d_{h-1}k_h<d_0^2$, $k_h\le \frac{d_0^2}{d_{h-1}}-1$. This is enough to finish the proof when $h=2$. Consider now that $h\ge 3$. By Proposition \ref{proposition_construction_A-M}, $S=S_{h-1}\oplus _{d_{h-1},k_h} \N$ where $S_{h-1}$ is the A-M semigroup of degree $\frac{d_0}{d_{h-1}}$ generated by the characteristic sequence $\left(\frac{d_0}{d_{h-1}},\frac{d_1}{d_{h-1}}k_1,\ldots,\frac{d_{h-2}}{d_{h-1}}k_{h-2},k_{h-1}\right)$.  So $\frac{d_{h-2}}{d_{h-1}}k_{h-1}<\left(\frac{d_0}{d_{h-1}}\right)^2$,  then $k_{h-1}\le\frac{d_0^2}{d_{h-1}d_{h-2}}-1$, hence $d_{h-1}k_{h-1}\le\frac{d_0^2}{d_{h-2}}-d_{h-1}$.
In general, using this process, for any $i\in\{3,\ldots ,h-1\}$, $S_{i}=S_{i-1}\oplus _{d_{i-1},k_i} \N$ where $S_{i-1}$ is the A-M semigroup of degree $\frac{d_0}{d_{i-1}}$ generated by the characteristic sequence $\left(\frac{d_0}{d_{i-1}},\frac{d_1}{d_{i-1}}k_1,\ldots,\frac{d_{i-2}}{d_{i-1}}k_{i-2},k_{i-1}\right)$. For any $i$, the condition that $S_{i-1}$ is an A-M semigroup implies that $\frac{d_{i-2}}{d_{i-1}}k_{i-1}< (\frac{d_0}{d_{i-1}})^2$, and $d_{i-1}k_{i-1}\le \frac{d_0^2}{d_{i-2}}-d_{i-1}$.
\end{proof}

Algorithm \ref{algoritmo2} provides a computational method to compute all A-M semigroups with fixed degree. Note that this algorithm supports parallel deployment.
\begin{algorithm}[h]
	\BlankLine
	\KwIn{$n\in\N\setminus\{0,1\}$.}
	\KwOut{The set of the characteristic sequences of A-M semigroups with degree $n$. }
	\BlankLine
	
	${\cal F}_n\leftarrow\emptyset$\;
	${\mathcal D}\leftarrow\big\{ d=(d_0,\ldots,d_h) \mid h\in \N,\, d \text{ is a sequence of divisors of }n \big\}$\;
    \While {${\mathcal D}\neq \emptyset$}
 	{
 	    $d=(d_0,\ldots,d_h)\leftarrow\text{First}(\cal D)$\;
 	    $\mathcal K\leftarrow \Big\{(k_1,\ldots ,k_h)\in \N^h\mid \gcd(k_1,n)=\cdots =\gcd(k_h,n)=1,\, 1\le k_1\le \frac{d_0}{d_1}-1,\, \text{and }d_{i-2}k_{i-1}+1\le d_{i}k_{i}\le  \frac{d_0^2}{d_{i-1}}-d_{i} \text{ for any }i=2,\ldots ,h
 	    \Big\}$\;
 		\ForAll{$(k_1,\ldots ,k_h)\in \mathcal K$}
     	{
	        ${\cal F}_n\leftarrow{\cal F}_n\cup \{(n,d_1k_1,\ldots,d_hk_h)\}$\;
 			
 		}
 		${\mathcal D}\leftarrow{\mathcal D}\setminus \{d\}$\;
 	}
	\Return ${\cal F}_n$ \;
\caption{Computation of the set of A-M semigroups with degree $n$.}\label{algoritmo2}
	
\end{algorithm}

Table \ref{tabla_AM_grado_8} shows all the A-M semigroups with degree eight. Note that, in this example, all the even integers in $[n-1,(n-1)(n-2)]$ are the conductor of some A-M semigroup of degree $n=8$. Throughout this work, we prove that this is true for all even degrees.
\begin{table}[]
    \centering
    \begin{tabular}{c|c|c}
 \text{Characteristic} & \text{Sequences of} & \text{Conductors}\\ 
  \text{sequences} & \text{divisors} & \\ \hline
(8,2,9) & (8,2,1) & 8 \\ \hline
 (8,2,11) & (8,2,1) & 10 \\ \hline
 (8,2,13) & (8,2,1) & 12 \\ \hline
 (8,3) & (8,1) & 14 \\ \hline
 (8,2,15) & (8,2,1) & 14 \\ \hline
 (8,2,17) & (8,2,1) & 16 \\ \hline
 (8,2,19) & (8,2,1) & 18 \\ \hline
 (8,2,21) & (8,2,1) & 20 \\ \hline
 (8,2,23) & (8,2,1) & 22 \\ \hline
 (8,4,9) & (8,4,1) & 24 \\ \hline
 (8,2,25) & (8,2,1) & 24 \\ \hline
 (8,2,27) & (8,2,1) & 26 \\ \hline
 (8,4,10,21) & (8,4,2,1) & 28 \\ \hline
 (8,5) & (8,1) & 28 \\ \hline
 (8,2,29) & (8,2,1) & 28 \\ \hline
 (8,4,10,23) & (8,4,2,1) & 30 \\ \hline
 (8,4,11) & (8,4,1) & 30 \\ \hline
 (8,2,31) & (8,2,1) & 30 \\ \hline
 (8,4,10,25) & (8,4,2,1) & 32 \\ \hline\
 (8,4,10,27) & (8,4,2,1) & 34 \\ \hline
 (8,6,25) & (8,2,1) & 36 \\ \hline
 (8,4,10,29) & (8,4,2,1) & 36 \\ \hline
 (8,4,13) & (8,4,1) & 36 \\ \hline
 (8,6,27) & (8,2,1) & 38 \\ \hline
 (8,4,10,31) & (8,4,2,1) & 38 \\ \hline
 (8,6,29) & (8,2,1) & 40 \\ \hline
 (8,4,14,29) & (8,4,2,1) & 40 \\ \hline
 (8,6,31) & (8,2,1) & 42 \\ \hline
 (8,4,14,31) & (8,4,2,1) & 42 \\ \hline
 (8,7) & (8,1) & 42 \\ \hline
 (8,4,15) & (8,4,1) & 42 \\ \hline
    \end{tabular}
    \caption{All the A-M semigroups of degree 8.}
    \label{tabla_AM_grado_8}
\end{table}

If $n>2$, $(d_0,\ldots,d_h=1)$ is a sequence of divisors of degree $n=d_0$ and we consider values $k_i$  as in Proposition \ref{AM-alg2}  then, by \eqref{conductor}, the conductor of the A-M semigroup
$S(n,d_1k_1,\ldots,d_hk_h)$ is
\begin{equation}\label{conductor_from_seq_div}
c(S(n,d_1k_1,\ldots,d_hk_h))=\sum_{i=1}^h(d_{i-1}-d_i)k_i-n+1.
\end{equation}

Note that the characteristic sequence given in \eqref{eq:divisors} is $(n,d_1k_1,\ldots,d_hk_h)$ for the maximum values of $k_i$ obtained in Proposition \ref{AM-alg2}. Moreover, as a consequence of that we have the following corollary.

\begin{corollary}
Let $n>2$ be a natural number. Fix a sequence of divisors $(d_0=n,d_1,\ldots, d_h)$ of $n$. The A-M semigroup of the form 
$S(n,d_1k_1,\ldots,d_hk_h)$ having the minimum conductor is given by $k_1=1$ and $k_i=\frac{d_{i-2}}{d_i}k_{i-1}+1$, for $i\in \{2,\ldots, h\}$ and   its conductor is $\sum_{i=0}^{h-2}d_i(d_{i+1}-1)$.
\end{corollary}

\begin{proof}

By \eqref{conductor_from_seq_div} and according to Proposition \ref{AM-alg2}, the minimum value of conductors of A-M semigroups of the form $S(n,d_1k_1,\ldots,d_hk_h)$ holds for the minimum values of $k_i$ for $i=1,\ldots , h$. 

For $h=1$, the sequence of divisors $(d_0,d_1,\ldots, d_h)$ is $(n,1)$. So, $S(n,1)=\N$ and its conductor is zero, which is equal to the empty sum $\sum_{i=0}^{-1}d_i(d_{i+1}-1)$.

Assume that $h\ge 2$. Put  $k_1=1$, and for any  $i\in \{2,\ldots ,h\}$, we put $k_i = \Big\lceil \frac{d_{i-2}k_{i-1}+1}{d_i}\Big\rceil = \frac{d_{i-2}k_{i-1}}{d_i}+1$. Notice that, since $\frac{d_{i-1}}{d_i}$ is a divisor of $\frac{d_{i-2}}{d_i}$, we have $\gcd\big(\frac{d_{i-1}}{d_i},\frac{d_{i-2}k_{i-1}}{d_i}+1\big)=1$. Thus, the integers $k_1=1$ and $k_i=\frac{d_{i-2}}{d_i}k_{i-1}+1$, for $i\in \{2,\ldots, h\}$, satisfy Proposition \ref{AM-alg2}, and then, $S(n,d_1k_1,\ldots,d_hk_h)$ is the A-M semigroup having the minimum conductor for the sequence of divisors $(d_0=n,d_1,\ldots, d_h)$.

Note that, for $k_1=1$, and $k_{q}=\frac{d_{q-2}k_{q-1}}{d_q}+1$, we have that $k_2= \frac{d_0}{d_2}+1$, and, in general, $k_q= \frac{1}{d_{q-1}d_q}\sum_{i=0}^{q-3}d_id_{i+1}+\frac{d_{q-2}}{d_q}+1$ for $q\in \{2,\ldots ,h\}$.

After \eqref{conductor_from_seq_div}, we get $c(S(n,d_1k_1,\ldots,d_hk_h))=\sum_{i=1}^h(d_{i-1}-d_i)k_i-n+1$. Thus,
\begin{eqnarray*}
c(S(n,d_1k_1,\ldots,d_hk_h))&=& \sum_{i=1}^h d_{i-1} k_i-\sum_{i=1}^h d_ik_i-n+1\\    
&=&  \sum_{i=1}^h d_{i-1} k_i-\sum_{i=2}^h d_ik_i-d_1 - n+1\\ 
&=&  \sum_{i=1}^h d_{i-1} k_i-\sum_{i=2}^h d_i\left(\frac{d_{i-2}}{d_i}k_{i-1}+1\right)-d_1 - n+1 \\
&=& d_{h-1}k_h - \sum_{i=0}^h d_i + 1 \\
&=& d_{h-1}\left( \frac{1}{d_{h-1}}\sum_{i=0}^{h-3} d_id_{i+1}  +d_{h-2} +1 \right) - \sum_{i=0}^h d_i + 1 \\
&=& \sum_{i=0}^{h-2} d_id_{i+1}  - \sum_{i=0}^{h-2} d_i = \sum_{i=0}^{h-2}d_i(d_{i+1}-1).
\end{eqnarray*}
\end{proof}

\begin{remark}
Given an integer $n=p_1^{\alpha_1}\cdots p_t^{\alpha_t}\ge 2$ with $p_1>p_2>\cdots >p_t$ prime integers, the maximum length of any sequence of divisors of $n$ is $\Lambda(n)=\sum_{i=1}^t\alpha_i$.
Moreover, fix $h$ a length of the sequences of divisors of $n$ allows us to provide a lower bound for the integers which could be realizable as the conductor of an A-M semigroup of degree $n$ with an associated sequence of divisors with length greater than or equal to $h$:  let us consider that $T_h=\min \{\sum_{i=0}^{h-2}d_i(d_{i+1}-1)\mid (d_0,d_1,\ldots, d_h) \text{ is a sequence of divisors of }n\}$, if $c<T_h$, then $c$ is not realizable as the conductor of an A-M semigroup of degree $n$ with an associated sequence of divisors with length greater than or equal to $h$.

An interesting point observed in all of our examples is that for any sequences $d,d' \in \cup_{h=1}^{\Lambda(n)} T_h$, if $d<_{\rm lex} d'$, their minimum conductors verify that $\sum_{i=0}^{h-2}d_i(d_{i+1}-1)< \sum_{i=0}^{h'-2}d'_i(d'_{i+1}-1)$, where the sequences of divisors are completed with zeros when compared with $<_{\rm lex}$ the lexicographical order. 
An example of this is shown for $n=105$ in Table \ref{table_cs_150_4}.
\begin{table}[h]
\centering
\begin{tabular}{c|c}
 \text{Sequence} & \text{Minimum}\\ 
  \text{of divisors} & \text{conductor}  \\ \hline
  $(105,1,0,0)$ & $0$ \\\hline
 $(105,3,1,0)$ & $210$ \\\hline
 $(105,5,1,0)$ & $420$ \\\hline
 $(105,7,1,0)$ & $630$ \\\hline
 $(105,15,1,0)$ & $1470$ \\\hline
 $(105,15,3,1)$ & $1500$ \\\hline
 $(105,15,5,1)$ & $1530$ \\\hline
 $(105,21,1,0)$ & $2100$ \\\hline
 $(105,21,3,1)$ & $2142$ \\\hline
 $(105,21,7,1)$ & $2226$ \\\hline
 $(105,35,1,0)$ & $3570$ \\\hline
 $(105,35,5,1)$ & $3710$ \\\hline
 $(105,35,7,1)$ & $3780$ \\\hline
\end{tabular}
    \caption{List of sequences of divisors for degree $105$ and minimum conductor of the A-M semigroup associated.} 
    \label{table_cs_150_4}
\end{table}

We raise the following question: Is it true that for fixed $n\in \N$ and any $d,d' \in \cup_{h=1}^{\Lambda(n)} T_h$ with $d<_{\rm lex} d'$ the value of $\sum_{i=0}^{h-2}d_i(d_{i+1}-1)$ is at most $\sum_{i=0}^{h'-2}d'_i(d'_{i+1}-1)$?
\end{remark}

\section{Conductors of A-M semigroups of even degree}
\label{proof}

Let ${\cal F}_n$ be the set of Abhyankar-Moh semigroups of degree $n>2$. Let ${\cal E}_n=[n-1, (n-1)(n-2)]\cap 2\mathbb Z$. The cardinality of ${\cal E}_n$ is
\begin{equation}
\label{car} \sharp {\cal E}_n=\left\{\begin{array}{ll}
\frac{(n-1)(n-3)+2}{2}& \hbox{\rm when $n$ is odd}\\
 &\\
\frac{(n-1)(n-3)+1}{2}& \hbox{\rm when $n$ is even.}\\
\end{array}
\right.
\end{equation}

For any $c\in {\cal E}_n$, is there
$S\in {\cal F}_n$ such that $c(S)=c$? We will prove that the answer is positive when $n$ is even. However, if $n$ is odd then this is not true. Indeed, suppose that $S$ is an A-M semigroup of a prime degree $n>2$. In this case,
$S=S(v_0,v_1)$, where $n=v_0$ is greater than  $v_1$, $v_1>1$ and $\gcd(v_0,v_1)=1$. Hence,
\[{\cal F}_n=\{(n,v_1)\;:\; 2\leq v_1\leq n-1\}\]
and the cardinality of ${\cal F}_n$ is $n-2$. By (\ref{car}), the cardinality of ${\cal E}_n$ is $\frac{(n-1)(n-3)}{2}+1$. Observe that $\frac{(n-1)(n-3)}{2}+1> n-2$ for any $n>3$. So, for $n>3$ we conclude that there are values in ${\cal E}_n$ which are not realizable as  conductors of an A-M semigroup of degree $n$. Remember that the only A-M semigroup of degree 2 is $\mathbb N$. Suppose that $n$ is an even integer greater than 2. First, we will prove that for any $c\in [n-1, \frac{n^2-2}{2})\cap 2\N$ there is an A-M semigroup of degree $n$ and conductor $c$.

\begin{lemma}
Let $n\ge 4$ be an even integer. The only A-M semigroup $S$ of degree $n$ with $\mathrm c(S)=n$ is the semigroup generated by the characteristic sequence $(n,2,n+1)$.
\end{lemma}

\begin{proof}
The characteristic sequence $(n,2,n+1)$ determines an A-M semigroup with conductor $n$. Let us prove the uniqueness. Let $S$ be an A-M semigroup of degree $n=\mathrm c(S)$ determined by the characteristic sequence $(v_0=n,v_1,v_2,\ldots,v_h)$. By \cite[Proposition 1.2]{Ba-GB-P}, $\mathrm c(S)=\sum_{i=1}^h(n_i-1)v_i-v_0+1$  with $n_i>1$ for $1\le i\le h$. Suppose that $h\ge 3$. Since $v_1<n<v_2<v_3<\cdots <v_h$, we get $\mathrm c(S)\geq \sum_{i=1}^3(n_i-1)v_i-n+1 >n$. On the other hand, if  $h=1$ then $c(S)=nv_1-n-v_1+1$. Since $n>v_1\ge 3$, $\mathrm c(S)\ge 3n-n-v_1+1>n+1$.

Let us suppose that $h=2$, so $e_2=1$. Since $v_2\ge n+1$ then $\mathrm c(S)> (n_1-1)v_1+(n_2-2)n$. Observe that for  $n_2>2$ we get $\mathrm c(S)>n$. So, $n_2=e_1=2$, $n_1=n/2$ and $v_1$ is an even integer. If $v_1\ge 4$ then $\mathrm c(S)\ge 2n-2>n$. Hence $v_1=2$ and consequently $\mathrm c(S)=n$ if and only if $v_2=n+1$.
\end{proof}

\begin{lemma} \label{lema: even}
Let  $n\ge 4$ be an even integer. For any $q\in\big[0,\frac{(n-1)(n-2)+n}{4}-1\big)$, the semigroup generated by $(n,2,n+1+2q)$ is an A-M semigroup of conductor $n+2q$.
\end{lemma}

\begin{proof}
The characteristic sequence  $(n,2,n+1+2q)$ determines an A-M semigroup for any even integer $q< \frac{(n-1)(n-2)+n}{4}-1$ with sequence of divisors $(n,2,1)$ and $\mathrm c\big(S (n,2,n+1+2q) \big)=n+2q$.
\end{proof}

\begin{corollary}\label{coro: even}
For any even integer $n\ge 4$ and $c\in [n-1, \frac{n^2-2}{2})\cap 2\N$, there is at least an  A-M  semigroup of degree $n$ and conductor $c$. 
\end{corollary}

Now, the target will be to prove that, if $n\ge 4$ is an even integer and $c\in [\frac{n^2-2}{2},(n-1)(n-2)]\cap 2\N$, then there is an A-M semigroup of degree $n$ and conductor $c$. We will do it in several steps.

\begin{lemma}
\label{lema:AM 2r}
Let $n$ be a natural number such that $n=2r\geq 4$, where  $r$ is odd. Let $k\in \mathbb N$ be co-prime with $r$ and $1\leq k\leq r-1$. The sequence $(n,2k,v_2)$  determines an A-M semigroup of degree $n$ if and only if it verifies the next  conditions:
\begin{enumerate}
    \item $v_2$ is an odd number,
    \item $kn+1\leq v_2\leq 2r^2-1$.
\end{enumerate}
\end{lemma}
\begin{proof}
It follows directly from the definition of an A-M semigroup.
\end{proof}

Suppose that $(n=2r,2k,v_2)$ defines an A-M semigroup  of degree $n$, where $1\leq k\leq r-1$ with $\gcd(r,k)=1$, so $e_0=n=2r>e_1=2> e_2=1$. Denote this semigroup by $\mathcal S(k,v_2)$. The conductor of $\mathcal S(k,v_2)$  equals $(r-1)2k+v_2-2r+1$. Hence if $\mathcal S(k,v_2)$ and $\mathcal S(k,v_2')$ are two A-M semigroups of degree $n$ with $v_2<v'_2$ then $\mathrm c(\mathcal S(k,v_2))<\mathrm c(\mathcal S(k,v'_2))$.
If we fix $k$, varying $v_2$ we get A-M semigroups attaining the following values for their conductor
\[
\mathrm c(\mathcal S(k,v_2))\in [2kn-2(k+r)+2,kn-2(k+r)+2r^2]\cap 2\mathbb N.
\]
Denote by $I_l:=[2ln-2(l+r)+2,ln-2(l+r)+2r^2]\cap 2\mathbb N$ for any $1\leq l \leq r-1$. Note that for any $l_1,l_2\in \N$ with $l_1<l_2$, $I_{l_1}\cap I_{l_2}$ is the empty set if and only if $l_2\ge \frac{r^2+l_1r-l_1}{2r-1}$. Since $\frac{r^2+l_1r-l_1}{2r-1}>\frac{r+l_1}{2}$, we obtain that $I_{l_1}\cap I_{l_2}\neq \emptyset$ for every $l_1\in \N$ and $l_2\in (l_1,\frac{r+l_1}{2}]\cap \N$. We want to show that by also varying the values of $k$ we construct A-M semigroups covering all possible conductors. For that we prove the following.

\begin{lemma}
\label{lema:aritmetico}
Let $r,l\in \mathbb N$ where $r$ is odd, $l\leq r-2$ and $\gcd(l,r)=1$. There is $k\in \left( l,\frac{l+r}{2}\right]\cap \mathbb N $ coprime with $r$.
\end{lemma}

\begin{proof}
Let $m:=\min\left\{n\in \mathbb N \;:\; r-2^n \leq \frac{l+r}{2}\right\}.$ Notice that such an integer exists since $l \leq r-2$ by hypothesis. We claim that $k:=r-2^m \in \left( l,\frac{l+r}{2}\right]$ and it is coprime with $r$. Indeed, $\gcd(r-2^m,r)=\gcd(2^m,r)=1$, due to $r$ being odd. Moreover, by definition of $m$, $k\leq \frac{l+r}{2}$. Assume by contradiction that $k= r-2^m\leq l$, then $2r-2^m \leq l+r$, that is,  $ r-2^{m-1} \leq \frac{l+r}{2}, $ which contradicts the minimality of $m$. Hence, $k\in \left( l,\frac{l+r}{2}\right]\cap \mathbb N $ is as desired.
\end{proof}

Hence we conclude,
\begin{proposition}\label{n=2r}
Let $n$ be a natural number such that $n=2r$, where  $r>1$ is odd. For any even number $c$ with $n-1\leq c\leq (n-1)(n-2)$ there is an A-M semigroup of degree $n$ and conductor equals $c$.
\end{proposition}
\begin{proof}
It is  enough to consider the characteristic sequences $(n=2r, 2k, v_2)$ defining A-M semigroups. Indeed, take $L:=\{l\in [1,r-1]\cap \N\mid \gcd(l,r)=1\}=\{l_1, l_2, \ldots,l_s\}$, with $l_t<l_{t+1}$ for any $t\in\{1,\ldots,s\}$. Lemma \ref{lema:AM 2r} and its following discussion tell us that by considering this characteristic sequences we construct A-M semigroups whose conductors cover all the values in $I_{l_1}\cup \ldots \cup I_{l_s}$. Moreover, by Lemma \ref{lema:aritmetico}, we know that $I_{l_i}\cap I_{l_{i+1}}\neq \emptyset$ for every $i\in\{1, \ldots, s-1\}$. Hence, we cover all even integers from the minimum value in $I_{l_1}$ to the maximum value in $I_{l_s}$. Since $l_1=1$ and $l_s=r-1$, these values are $n$ and $(n-1)(n-2)$, respectively. This concludes the proof.
\end{proof}

\begin{remark} \label{n=4} After Algorithm \ref{algoritmo1}, we get that the A-M semigroups of degree $4$ are $S_1=S(4,2,5)$, $S_2=S(4,3)$ and $S_3=S(4,2,7)$, where
$c(S_1)=4$ and $c(S_2)=c(S_3)=6$. Hence any $c\in \mathcal E_4$ is the conductor of an A-M semigroup of degree 4. Observe that the sequence of divisors of $S_1$ and $S_3$ is $(4,2,1)$ and the sequence of divisors of $S_2$ is $(4,1)$.
\end{remark} 

\begin{proposition}\label{n par}
Let $n$ be a natural number such that $n=2^{k}r$, where $k\geq 1$ and  $r>1$ is odd. For any even number $c$ with $n-1\leq c\leq (n-1)(n-2)$ there is an A-M semigroup of degree $n$ and conductor equals $c$.
\end{proposition}

\begin{proof}
We use induction over $k$.
Notice that for $k=1$, the statement holds by Proposition \ref{n=2r}. 

Let us assume that $k\geq 2$ and that the proposition holds for $k-1$. In addition, Remark \ref{n=4} already covers the case $n=4$, thus we can further assume that $n\geq 8$. 

We first observe that, by Corollary \ref{coro: even}, for every even $c\in I_1 := [n-1, \frac{n^2-2}{2})$, there exists an A-M semigroup of degree $n$ and conductor $c$. Thus, it remains to prove that the same is true for any $c\in \left[\frac{n^2-2}{2}, (n-1)(n-2) \right]$. 

On the one hand, as explained in Lemma \ref{lema: even}, $S'=\langle \frac{n}{2}, 2, \frac{n}{2}+1\rangle$ is an A-M semigroup with conductor $c(S')=\frac{n}{2}$. One can check that the gluing $S=S'\bigoplus_{2,d} \mathbb N = \langle n, 4, n+2, f\rangle$ is an A-M semigroup for every odd number $f \in \left[2n+5, \frac{n^2}{2}-1\right]$. Furthermore, \eqref{CondGSI} tells us that
$$c(S) = 2\frac{n}{2} + f -1 = n+f-1.$$
Hence, through this gluing we have shown that for any even number $c $ belonging to $I_2:= \left[3n+4, \frac{n^2}{2}+n-2\right]$, the semigroup $S=\langle n, 4, n+2, c+1-n \rangle$ is an A-M semigroup with conductor $c$.

On the other hand, by the induction hypothesis, for $m=\frac{n}{2}=2^{k-1}r$, we can guarantee that, for every even number $c' \in [m-1, (m-1)(m-2)]$, there exists an A-M semigroup $S'_{c'}$ of degree $m$ and conductor $c'$. It is not hard to check that the gluing
$$ S(c',f_i) = S'_{c'}\oplus_{2,f_i}\mathbb N $$
is an A-M semigroup of degree $2m=n$, for both values: $f_1=\frac{n^2}{2}-1$ and $f_2=\frac{n^2}{2}-3$. Besides, again by \eqref{CondGSI}, the conductor of $S(c',f_i)$ is
$$ c(c',f_i) = 2c'+f_i-1. $$
Now, notice the relations
$$ c(c',f_2) = c(c',f_1) -2$$
and 
$$ c(c'-2,f_1)= c(c', f_2) -2, $$
which imply that, by making this gluing and recursively decrease the value of $c'$ alternating it with both $f_1$ and $f_2$, we can assure the existence of an A-M of degree $n$ and conductor $c$, for every even number from the maximum possible value, that is,
$$c((m-1)(m-2), f_1) = 2\left(\frac{n}{2}-1 \right)\left(\frac{n}{2}-2\right) + \frac{n^2}{2}-1 = (n-1)(n-2), $$
and all the way down to
$$c(m,f_2) = 2\frac{n}{2} + \frac{n^2}{2}-3-1 =\frac{n^2}{2}+n-4. $$

Overall, we have shown that, for any even number $c \in I_1\cup I_2\cup I_3$ there exists an A-M semigroup of degree $n$ and conductor $c$, where $I_3 = \left[\frac{n^2}{2}+n-4, (n-1)(n-2)\right]$. Clearly, $I_2$ and $I_3$ overlap, and we have that $I_1 \cup I_2 \cup I_3=[n-1,(n-1)(n-2)]$ as long as $n\geq 8$. Thus, the statement follows.
\end{proof}

Let $p> 1$ be a prime number and $k$ be an integer greater than or equal to two. Note that, for every A-M semigroup $S$ of degree $n=p^k$, $\mathrm c(S)$ is a multiple of $p-1$ (see equation \eqref{conductor}). So, $\mathrm c(S) \in [n-1,(n-1)(n-2)]\cap (p-1)\N$. 

Now, we will study the A-M semigroups of degree $n=p^k$, where $p$ is a prime number and $k\in \mathbb N$, $k>1$. For any natural number  $k_1$ with $1\leq k_1\leq p^{k-1}-1$, we define $I_{k_1}:=[1,p^{2k-1}-p^kk_1-1]$, and
\[
\mathcal A_{k_1}:=\Big\{(p^{k}-1)(pk_1-1)+i(p-1)\;:\;i\in I_{k_1}\text{ and }\gcd(p,i)=1\Big\}.
\]

In the following proposition, we prove that for a fixed $n=p^k$ with $p$ a prime number, every value in $[n-1,(n-1)(n-2)]\cap (p-1)\N$ is the conductor of some A-M semigroup of degree $n$. Furthermore, the conductors of the A-M semigroups of degree $n=p^k$ are provided explicitly for any prime integer $p$.

\begin{proposition}\label{n=p^k}
Let $p\ge 2$ be a prime number, and $k\ge 2$ an integer. Any conductor of an A-M semigroup with degree $n=p^k$ is obtained from, at most, two types of sequences of divisors:
\begin{enumerate}
    \item For $p$ a prime odd number, the sequences are $(p^k,1)$ and $(p^k,p,1)$;
    \item For $p=2$ and $k=2$, the sequence is $(4,2,1)$;
    \item For $p=2$ and $k=3$, the sequences are $(8,2,1)$ and $(8,4,2,1)$;
    \item For $p=2$ and $k\ge 4$, the sequence is $(2^k,2,1)$.
\end{enumerate}
Moreover, for any odd prime number $p$, the set of the conductors of A-M semigroups of degree $n$ is
\[
\left(\bigcup_{i=1,\, \gcd(p,i)=1}^{p^k-2}\big\{i(p^k-1)\big\}\right) \bigcup \left(
\bigcup_{k_1=1,\, \gcd(p,k_1)=1}^{p^{k-1}-1} \mathcal A_{k_1}\right).
\]
\end{proposition}

\begin{proof}
Let $(p^k,pk_1,v_2)$ be a characteristic sequence where $\gcd(p,k_1)=\gcd(p,v_2)=1$. Note that the semigroup $S(p^k,pk_1,v_2)$  is an A-M semigroup if and only if $p^{k}k_1<v_2<p^{2k-1}$. 
So, we assume that $k_1\in [1,p^{k-1}-1]$ and $v_2\in[p^k k_1+1,p^{2k-1}-1]$, with $k_1$ and $v_2$ coprime with $p$. Consider $v_2=p^k k_1+i$ with $i\in I_{k_1}=[1,p^{2k-1}-p^k k_1-1]$ and such that $i\neq 0\;\mod p$. 

Thus, $\mathrm c(S)= (p^{k}-1)(pk_1-1)+i(p-1)$. Denote $\mathrm c(p^k,pk_1,p^k k_1+i)= (p^{k}-1)(pk_1-1)+(p-1)i$. This even value belongs to $C(k_1)\cap (p-1)\N$ where $C(k_1)= [(p^{k}-1)(pk_1-1)+(p-1),(p^{k}-1)(pk_1-1)+(p-1)(p^{2k-1}-p^k k_1-1)]$, for any $i\in I_{k_1}$. Note that, for $p=2$, $\mathrm c(2^k,2k_1,2^k k_1+i)$ take all the even numbers in $C(k_1)$. Besides, for any integer $k_1'$ such that $k_1<k_1'$, we have that
\begin{equation}\label{eq:1 de p^k}
(C(k_1)\cup C(k'_1))\cap 2\N \cap (p-1)\N= [\min C(k_1),\max C(k'_1)]\cap 2\N \cap (p-1)\N
\end{equation}
if and only if $\min C(k'_1)\le \max C(k_1)+\max \{2,p-1\}$.

Assume that $p\ge 3$ and suppose $\gcd(p,k_1+1)=1$. In that case, the condition \eqref{eq:1 de p^k} holds for $k_1+1\le p^{k-1}-1$ iff $\min C(k_1+1)\le \max C(k_1)+p-1$, that is, iff $(p-1)(p^{2k-1}-p^k k_1-1)-p(p^k-1)\ge 0$. Since $k_1\le p^{k-1}-2$,
\begin{multline*}
(p-1)(p^{2k-1}-p^k k_1-1)-p(p^k-1) \ge (p-1)(2p^k -1)-p(p^k-1)
\\
= p^k(p-2)+1> 0
\end{multline*}
Thus, the condition \eqref{eq:1 de p^k} holds when $\gcd(p,k_1+1)=1$. 
In the other case, $k_1+1$ and $p$ are no coprime numbers, we have that  $\gcd(p,k_1+2)=1$, and the inequality $\min C(k_1+2)\le \max C(k_1)+p-1$ is equivalent to $(p-1)(p^{2k-1}-p^k k_1-1)-2p(p^k-1)\ge 0$. Again, since $k_1+2\le p^{k-1}-1$, $(p-1)(p^{2k-1}-p^k k_1-1)-2p(p^k-1) \ge p^k(p-3)+p+1 > 0$. We conclude that, for $p\ge 3$, the condition \eqref{eq:1 de p^k} always holds. That means that, for every $k_1\in [1,p^{k-1}-1]$ and $i\in I_{k_1}$ with $i\mod p\neq 0$, $\mathrm c(p^k,pk_1,p^k k_1+i)= (p^{k}-1)(pk_1-1)+i(p-1)\in \mathcal A=[\min C(1),\max C(p^{k-1}-1)]\cap (p-1)\N=[n(p-1),(n-1)(n-2)]\cap (p-1)\N$. Note that $\mathcal A$ is the union of the disjoint sets
\[
\mathcal A_{k_1}= \bigcup_{k_1=1,\, \gcd(p,k_1)=1}^{p^{k-1}-1}\Big\{c(p^k,pk_1,p^k k_1+i)\mid i\in I_{k_1}, \text{ and }\gcd(p,i)=1\Big\},
\]
and
\[
\mathcal A'_{k_1}= \bigcup_{k_1=1,\, \gcd(p,k_1)=1}^{p^{k-1}-1}\Big\{c(p^k,pk_1,p^k k_1+i)\mid i\in I_{k_1}, \text{ and }\gcd(p,i)\neq 1\Big\}.
\]
The elements of $\mathcal A'_{k_1}$ are not the conductors of any A-M semigroup of degree $n$. Suppose there exists an A-M semigroup $S'$ of degree $n$ and such that its conductor belongs to $\mathcal A'_{k_1}$. If the characteristic sequence of $S'$ is $(p^k,v_1)$ with $\gcd(p,v_1)=1$, then $\mathrm c(S')= (p^k-1)(v_1-1) =(p^{k}-1)(pk_1-1)+i(p-1)$, and $(p^k-1)v_1=(p^{k}-1)pk_1+i(p-1)$. But it is not possible since $\gcd(p^{k}-1,p)=\gcd(p,v_1)=1$. If the length of the characteristic sequence of $S'$ is greater than or equal to three, by Proposition \ref{proposition_construction_A-M} and Corollary \ref{CondGSI}, there exist an A-M semigroup $S''$, $q\in \N$, and $f\in \N$ with $1<q<k$ and $\gcd(p,f)=1$ such that $S'=S''\oplus _{p^q,f} \N$. Thus, $\mathrm c(S')=p^q\mathrm c(S'')+(p^q-1)(f-1)=(p^{k}-1)(pk_1-1)+i(p-1)$, and $(p^q-1)f= (p^k-1)pk_1+i(p-1)-p^q\mathrm c(S'')-p^q(p^{k-q}-1)$. Since $\gcd(p^{q}-1,p)=\gcd(p,f)=1$, it does not hold.

Note that, by Proposition \ref{proposition_construction_A-M} and Corollary \ref{CondGSI}, $n(p-1)$ is the smallest conductor of an A-M semigroup of degree $n$ that can be obtained from characteristic sequences of length greater than or equal to three. Moreover, the set of the conductors of the A-M semigroups of degree $n$ with a characteristic sequence of length two is $\cup_{v_1=2,\, \gcd(p,v_2)=1}^{p^k-1}\{(v_1-1)(p^k-1)\}$.

Summarising, we have just proved that, for every $p\ge 3$ prime odd number and $k\ge 2$, the conductor of any A-M semigroup with degree $n=p^k$, is equal to $i(p^k-1)$ with $i\in[1,p^k-2]$ satisfying $i\mod p\neq 0$, or is equal to $(p^{k}-1)(pk_1-1)+i(p-1)$ where $k_1\in [1,p^{k-1}-1]$ and $i\in I_{k_1}$ such that $i\mod p\neq 0$. For the first type, their sequences of divisors are $(p^k,1)$, and $(p^k,p,1)$ for the second one.

For $p=2$, $\gcd(2,k_1+1)\neq 1$, but $\gcd(2,k_1+2)=1$. Hence, the conditions \eqref{eq:1 de p^k} holds if and only if $2^{2k-1}-2k_1-2^{k+2}+3 \ge 0$. Using the upper bound of $k_1$, $2^{2k-1}-2k_1-2^{k+2}+3\ge 2^{2k-1}-5\cdot 2^{k}+5$. It is easy to prove that $2^{2k-1}-5\cdot 2^{k}+5\ge 0$ for any $k\ge 4$. Therefore, for $k\ge 4$, every even number in $[n-1,(n-1)(n-2)]$ is realizable as the conductor of an A-M semigroup with the sequence of divisors $(2^k,2,1)$. The particular cases $k=2$, and $k=3$ are showed in Remark \ref{n=4} and Table \ref{tabla_AM_grado_8}, respectively.
\end{proof}

By Proposition \ref{n par} and Proposition \ref{n=p^k}, for $p=2$, we have,

\begin{maintheorem}\label{teorema: n par}
Let $n>2$ be an even natural number. For any even number $c$ with $n-1\leq c\leq (n-1)(n-2)$, there is an A-M semigroup of degree $n$ and conductor equals $c$.
\end{maintheorem}

\noindent {\bf Open question:} Characterize the  values of conductors of A-M semigroups of odd degree. Note that Proposition \ref{n=p^k} solves this question for degree $p^k$ with $p$ any prime integer.\\

\noindent {\bf Funding}. The second and fourth-named authors were supported partially by Junta de Andaluc\'{\i}a research groups FQM-343. The first-named author  was supported by the grant PID2019-105896GB-I00 funded by MCIN/AEI/10.13039/501100011033. The first and third-named authors were supported partially by MACACO (ULL research project).

\medskip
\noindent
{\small Evelia Rosa Garc\'{\i}a Barroso / Luis Jos\'e Santana S\'anchez\\
Departamento de Matem\'aticas, Estad\'{\i}stica e I.O. \\
Secci\'on de Matem\'aticas, Universidad de La Laguna\\
Apartado de Correos 456\\
38200 La Laguna, Tenerife, Spain\\
e-mail: ergarcia@ull.es / lsantans@ull.es}

\medskip

\noindent {\small Juan Ignacio Garc\'{\i}a-Garc\'{\i}a \\
Departamento de Matem\'aticas/INDESS (Instituto Universitario para el Desarrollo Social Sostenible)\\
Universidad de C\'adiz\\
 E-11510 Puerto Real, C\'adiz, Spain\\
e-mail: ignacio.garcia@uca.es}

\medskip

\noindent {\small Alberto Vigneron-Tenorio\\
Departamento de Matem\'aticas/INDESS (Instituto Universitario para el Desarrollo Social Sostenible)\\
Universidad de C\'adiz\\
E-11406 Jerez de la Frontera, C\'adiz, Spain\\
e-mail: alberto.vigneron@uca.es}
\end{document}